\documentclass[reqno,11pt]{amsart}
\usepackage{geometry}
\usepackage{mathtools,amssymb,amsthm,mathrsfs,color,lineno,paralist,graphicx,float}
\usepackage[colorlinks,
linkcolor=blue,
anchorcolor=green,
citecolor=blue, 
]{hyperref}
\usepackage{tikz}
\usepackage{graphicx}
\usetikzlibrary{arrows.meta}

\usepackage{enumitem}
\usepackage[T1]{fontenc}
\usepackage[utf8]{inputenc}
\usepackage{subfig} 
\usepackage[justification = centering, labelsep =period]{caption}

\usepackage{calc}

\definecolor{bleu1}{RGB}{0,57,128}
\def\bleu1{\color{bleu1}}

\usepackage{etoolbox}
\patchcmd{\section}{\normalfont}{\normalfont \bleu1}{}{}
\patchcmd{\subsection}{\normalfont}{\normalfont \bleu1}{}{}
\patchcmd{\subsubsection}{\normalfont}{\normalfont \bleu1}{}{}

\newtheorem{proposition}{Proposition}[section]
\newtheorem{theorem}{Theorem}[section]
\newtheorem{definition}{Definition}[section] 
\newtheorem{lemma}{Lemma}[section]

\newtheorem{corollary}{Corollary}[section]

\newtheorem{problem}{Problem}[section]

\newcommand{\Z}{{\mathbb Z}}
\newcommand{\C}{{\mathbb C}}
\newcommand{\R}{{\mathbb R}}

\def\0{{\mathbf 0}}

\usepackage{tikz,tikz-3dplot}
\tdplotsetmaincoords{80}{45}
\tdplotsetrotatedcoords{-90}{180}{-90}
\tikzset{surface/.style={draw=blue!70!black, fill=blue!40!white, fill opacity=.6}}

\usepackage{pgfplots}
\usepgfplotslibrary{polar}
	\usepgfplotslibrary{polar}
\pgfplotsset{compat=1.17}
\makeatletter
\tikzset{reuse path/.code={\pgfsyssoftpath@setcurrentpath{#1}}}
\makeatother

\begin{document}

\author{Zhenfu Wang}
\address{
	Chern Institute of Mathematics and LPMC, Nankai University, Tianjin 300071, China
}
\email{zhenfuwang@mail.nankai.edu.cn}

 \author{Bei Zhang}
 \address{
 Chern Institute of Mathematics and LPMC, Nankai University, Tianjin 300071, China} 
 \email{beizhang@nankai.edu.cn}

\author{Qi Zhou}
\address{
	Chern Institute of Mathematics and LPMC, Nankai University, Tianjin 300071, China
}
\email{qizhou@nankai.edu.cn}

\begin{abstract}
 We establish an ultimate Ishii--Pastur theorem for whole-line ergodic
block Jacobi operators. We prove that, for almost every realization,
the restriction of a maximal spectral measure to the region where the
smallest nonnegative Lyapunov exponent is strictly positive is carried
by a Borel set of zero capacity. In the scalar case, this
settles the whole-line problem formulated by Damanik and Fillman
\cite[Problem~4.7.18]{Damanik-Fillman-book}.
\end{abstract}

\title[]{The Ultimate Ishii--Pastur Theorem for Whole-Line Ergodic Block Jacobi Operators}

\maketitle{\Large }

\section{introduction}

The Ishii--Pastur--Kotani theorem \cite{Ishii,Kotani,Pastur} is one of the fundamental links
between dynamics and spectral theory for one-dimensional ergodic
Schr\"odinger operators  \cite{Damanik-Fillman-book,DF2}. Let $(\Omega,\mathbb P,T)$ be an invertible
ergodic probability dynamical system,  that is, $\Omega$ is a compact metric space, $T:\Omega\to\Omega$ is a
homeomorphism, and $\mathbb P$ is a $T$-invariant ergodic Borel probability
measure. Consider  one-dimensional ergodic
Schr\"odinger operators on $\ell^2(\Z):$
\begin{equation}\label{scalar}
(H_\omega u)_n
=
u_{n+1}+u_{n-1}+V(T^n\omega)u_n,
\qquad n\in\mathbb Z.
\end{equation}
Let $L(E)$ denote the Lyapunov exponent of the associated Schr\"odinger
cocycle and set
$$
\mathcal Z=\{E\in\mathbb R:L(E)=0\}.
$$
The Ishii--Pastur--Kotani theorem identifies the almost-sure
absolutely continuous spectrum as
$\Sigma_{\mathrm{ac}}
=
\overline{\mathcal Z}^{\,\mathrm{ess}}.$
In particular, the restriction of a maximal spectral measure to the
positive-Lyapunov region $\mathbb R\setminus\mathcal Z$ is singular
with respect to Lebesgue measure and is therefore carried by a Borel
set of Lebesgue measure zero.

This naturally leads to a finer question: how ``small'' can a supporting set in the positive Lyapunov exponent regime be chosen?
This yields the following problem put forward by Damanik and Fillman \cite{Damanik-Fillman-book}.

\begin{problem}\label{open-problem}\cite[Problem~4.7.18]{Damanik-Fillman-book}
Prove the ultimate Ishii–Pastur theorem for whole-line ergodic operators \eqref{scalar}. More precisely, show that the restriction of spectral measure to $\mathbb{R}\backslash\mathcal{Z}$ enjoys a support of zero capacity\footnote{ For the precise definition, see Definition \ref{cap}.} for almost every $\omega\in\Omega$.
\end{problem}

We give an affirmative answer to this problem:

\begin{theorem}[Ultimate Ishii--Pastur theorem]\label{thm:scalar}
For $\mathbb P$-almost every $\omega\in\Omega$, the restriction of any maximal spectral
measure of $H_\omega$ to $\{E:L(E)>0\}$ is carried by a Borel set
of zero capacity.
\end{theorem}

While Theorem \ref{thm:scalar}  resolves the open problem formulated by Damanik and Fillman, our methods go beyond the scalar setting. We establish the
whole-line ultimate Ishii--Pastur theorem for arbitrary finite-width
ergodic strip operators, and more generally for ergodic block Jacobi
operators with invertible hopping matrices. 
To make this precise, 
 let
$A:\Omega \rightarrow \operatorname{GL}(m,\mathbb C),B:\Omega\rightarrow \operatorname{Her}(m)
$
be measurable functions satisfying
\begin{equation*}
\|A\|_{L^\infty(\Omega)}+\|B\|_{L^\infty(\Omega)}+\int_\Omega\log^+\|A(\omega)^{-1}\|d \mathbb P(\omega)<\infty.
\end{equation*}
For $n\in\mathbb Z$, set
$
A_n(\omega):=A(T^n\omega),
B_n(\omega):=B(T^n\omega).
$
The associated whole-line block Jacobi operator $\mathcal H_\omega$ acts on
$\ell^2(\mathbb Z,\mathbb C^m)$ by
$$
(\mathcal H_\omega u)_n=A_{n-1}(\omega)^*u_{n-1}+B_n(\omega)u_n+A_n(\omega)u_{n+1},
\quad n\in\mathbb Z.
$$
 These operators arise naturally in
analytic theory of matrix orthogonal polynomials \cite{DPS}, XY spin chain models \cite{HSS}, and one can rewrite the dynamically defined finite-range operators as block Jacobi operators \cite{Puig}.

For each $E\in\mathbb R$, the eigenvalue equation $\mathcal H_\omega u=Eu$ induces
the $2m$-dimensional cocycle $(T,T_E)$, where
\begin{equation*}
T_E(\omega)=
\begin{pmatrix}
A(\omega)^{-1}(EI_m-B(\omega))
&-A(\omega)^{-1}A(T^{-1}\omega)^*\\
I_m&0
\end{pmatrix}.
\end{equation*}
Denote its  Lyapunov exponents by $$
L_1(E)\geq\cdots\geq L_m(E)\geq -L_m(E)\geq \cdots \geq -L_1(E)$$ (see Section \ref{cocycle} for the precise definition).
Define the region of
strictly positive Lyapunov exponents by
\begin{equation}\label{equ:+}
S_+:=\left\{E\in\mathbb R:L_m(E)>0\right\}.
\end{equation}

Denote $\mu_\omega$ be the maximal spectral measure of $\mathcal H_\omega$. Our main
result is the following whole-line ultimate Ishii--Pastur theorem for
ergodic block Jacobi operators.

\begin{theorem}\label{thm:main}
For $\mathbb P$-a.e. $\omega\in\Omega$, there exists a Borel set
$Q_\omega\subset\mathbb R$ of  capacity zero such that
$\mu_\omega(S_+\backslash Q_\omega)=0.$
\end{theorem}

Moreover, we have the following: 

\begin{corollary}
\label{cor:uniform}
If $(\Omega,\mathbb P,T)$ is uniquely ergodic and the Jacobi
coefficients $A(\cdot),B(\cdot)$ are continuous, then the conclusion of
Theorem~\ref{thm:main} holds for every $\omega\in\Omega$.
\end{corollary}

\subsection{Previous results and whole-line difficulty}

For  scalar half-line  and whole-line quasi-periodic Schr\"odinger operators,
Jitomirskaya and Last \cite{JL99-Acta,JL00-CMP} used power-law
subordinacy theory to show that, in the positive-Lyapunov regime, the
spectral measure is carried by a set of zero Hausdorff dimension. In
the general ergodic setting, Simon \cite{Simon07-IPI} proved the
stronger zero capacity conclusion for scalar  \textbf{half-line}
Schr\"odinger operators: the restriction of the spectral measure to
the positive-Lyapunov region is carried by a Borel set of zero
capacity. This is sharper, as every set of zero
capacity has zero Hausdorff dimension \cite{Tsu}.
Goldsheid and Sodin \cite{GS24-IMRN} subsequently established the
corresponding \textbf{half-line} result for ergodic Schr\"odinger operators with
matrix-valued potentials. We also recall that the matrix-valued
extension of the Ishii--Pastur--Kotani theorem had previously been
developed by Kotani and Simon \cite{KS}.

Simon \cite{Simon07-IPI} also stated a whole-line zero capacity
analogue. There is, however, a subtle distinction between the
asymptotic information supplied by logarithmic potential theory and
that required by the subsequent whole-line argument. As Goldsheid and
Sodin \cite{GS24-IMRN} remarked in their discussion of Simon's
half-line theorem, the upper envelope theorem yields the
quasi-everywhere upper-limit asymptotics
\begin{equation}\label{eq:limsup-asymptotics}
\limsup_{n\to\infty}
\frac{1}{n}\log\|\Phi_n(E,\omega)\|
=
L(E),
\end{equation}
rather than, by itself, the existence of the full limit
\begin{equation}\label{eq:full-limit}
\lim_{n\to\infty}
\frac{1}{n}\log\|\Phi_n(E,\omega)\|
=
L(E),
\end{equation}
where $\Phi_n(E,\omega)$ is the transfer matrix.
The weaker statement \eqref{eq:limsup-asymptotics} is sufficient for
the half-line spectral-measure argument, and hence does not affect
the half-line conclusion. In the whole-line proof, by contrast,
the deterministic Ruelle--Oseledets step uses full asymptotic control
of two independent fundamental solutions. Consequently, the
upper-limit statement alone does not directly justify the whole-line
passage.

The distinction between \eqref{eq:limsup-asymptotics} and
\eqref{eq:full-limit} is not merely technical.
Jitomirskaya and Liu \cite{JL18} exhibited natural quasiperiodic
examples with non-regular transfer-matrix growth, for which, at
suitable phases,
\begin{equation*}
\liminf_{n\to\infty}
\frac{1}{n}\log\|\Phi_n(E,\omega)\|
<
\limsup_{n\to\infty}
\frac{1}{n}\log\|\Phi_n(E,\omega)\|
=
L(E).
\end{equation*}
In the random setting, Gorodetski and Kleptsyn
\cite{GK21-Advances} studied parameter-dependent products of
independent random matrices, including those arising from
one-dimensional Anderson models, and showed that non-Lyapunov
behaviour may occur on a dense $G_\delta$ set of parameters of zero
Hausdorff dimension. The capacity of the corresponding
exceptional set remains unknown. Motivated by this question, Kleptsyn
and Quintino \cite{KQ22-Potential} studied model
$G_\delta$ sets and established capacity phase transitions that
support the conjecture that quasi-everywhere convergence in
\eqref{eq:full-limit} may fail, or equivalently that the actual
non-Lyapunov exceptional set may have positive capacity.

The principal difficulty is the whole-line geometry. A half-line
boundary condition restricts the admissible initial data to a fixed
subspace, whereas a whole-line solution must be controlled
simultaneously at both $+\infty$ and $-\infty$. For each fixed energy,
the Oseledets theorem gives the relevant stable and unstable directions
for almost every realization, but the exceptional set of realizations
may depend on the energy. It therefore does not provide a fixed
realization for which the required dichotomy holds outside a set of
capacity zero. Moreover, the logarithmic-potential
argument used in Simon's whole-line approach yields only
quasi-everywhere $\limsup$ asymptotics, rather than the full limits
needed for the subsequent deterministic Ruelle--Oseledets argument.

\subsection{Main ideas of the proof}
The distinction between whole-line and half-line spectral measures is
genuine.  Indeed, Levi \cite{Levi} constructed a whole-line
Schr\"odinger operator whose spectral measure has Hausdorff dimension
one, while the spectral measures of all its half-line restrictions
have Hausdorff dimension zero.  Thus, whole-line spectral information
cannot, in general, be recovered by combining separate conclusions
for the two half-line restrictions.
To address this difficulty, inspired by an observation of Gonz\'alez
and Sadel \cite[Remarks~(g)]{GS}, we flip the negative half-line onto the
nonnegative half-line via the unitary map
$$
(Uu)_n=\binom{u_n}{u_{-n-1}},\qquad n\geq0.
$$
This identifies the whole-line block Jacobi operator
$\mathcal H_\omega$ with a half-line block Jacobi operator
$\mathcal J_\omega$ of block size $2m$.  Crucially, this procedure is
not a decoupling into two half-line restrictions: the coupling across
the origin is retained, and $\mathcal J_\omega$ is unitarily
equivalent to the original whole-line operator.

In contrast to the
constant-hopping random setting considered in \cite{GS}, we carry out
the construction for general ergodic block Jacobi operators and use it
to identify the corresponding finite-volume restrictions and maximal
spectral measures, although the flipped family is no longer ergodic.
The two-sided evolution is thereby encoded by a single Dirichlet
fundamental matrix $D_n(E,\omega)$, and control of all whole-line
initial data is reduced to estimating the smallest singular value of
$D_n(E,\omega)$. Although the flipped half-line family is no longer ergodic, its
finite-volume Dirichlet restrictions are unitarily equivalent to
symmetric finite-volume restrictions of the original whole-line
ergodic operator. This allows us to retain the finite-volume spectral
asymptotics required in the argument.

The remaining, genuinely matrix-valued difficulty is to prove that,
outside a set of capacity zero,
\begin{equation}\label{eq:min-singular-lower-intro}
\limsup_{n\to\infty}
\frac1n\log s_{2m}(D_n(E,\omega))
\geq L_m(E).
\end{equation}
We establish this estimate through the identity
$$
s_{2m}(D_n(E,\omega))
=
\frac{|\det D_n(E,\omega)|}
     {\|\wedge^{2m-1}D_n(E,\omega)\|}.
$$
The determinant is controlled through the finite-volume spectral
asymptotics inherited from the original ergodic operator, while the
exterior power is bounded by expressing $D_n(E,\omega)$ in terms of
the forward and backward transfer matrices and applying Craig--Simon
estimates \cite{CS83-Duke} to their exterior powers. The difference between the two
resulting volume-growth rates is precisely the smallest positive
Lyapunov exponent $L_m(E)$, yielding
\eqref{eq:min-singular-lower-intro}.

Finally, spectral-almost every energy admits a nonzero initial vector
whose corresponding Dirichlet solution grows at most polynomially.
When $L_m(E)>0$, this is incompatible with the subsequential
exponential expansion implied by
\eqref{eq:min-singular-lower-intro}. It follows that the spectral
measure in the positive-Lyapunov region is carried by the exceptional
set of zero capacity.

\subsection{Organization of the paper}

Section~\ref{sec:preliminaries} reviews the required facts about
ergodic block Jacobi operators and logarithmic
potential theory. 
Section~\ref{flip} develops the flipping construction and records
the resulting infinite-volume, finite-volume, and spectral-measure
identifications. In Section~\ref{sec:Dirichlet}, we combine determinant
asymptotics with exterior-power bounds to obtain a quasi-everywhere
lower bound on the smallest singular value of the flipped Dirichlet
fundamental matrix. This estimate, together with the Schnol theorem,
completes the proof of the main result.

\section{preliminaries}\label{sec:preliminaries}
\subsection{Cocycle}\label{cocycle}
Let $(\Omega,\mathbb P,T)$ be an invertible ergodic probability
dynamical system. Let
$A:\Omega\rightarrow \operatorname{GL}(m,\mathbb C)$
be a measurable map satisfying
\begin{equation*}
\log^+\|A(\cdot)\|\in L^1(\Omega,\mathbb P),
\quad
\log^+\|A(\cdot)^{-1}\|\in L^1(\Omega,\mathbb P).
\end{equation*}
A cocycle $(T, A)$ is a linear skew product:
$$
(T,A):\left\{
\begin{array}{rcl}
	\Omega\times \C^{m} &\to& \Omega \times \C^{m}\\
	(\omega,v) &\mapsto& (T\omega,A(\omega)v)
\end{array}
\right.  .
$$
For $n\in\mathbb{Z}$, $A_n$ is defined by $(T,A)^n=(T^n,A_n).$ Thus $A_{0}(\omega)=id$,
\begin{equation*}
	A_{n}(\omega)=\prod_{j=n-1}^{0}A(T^{j}\omega)=A(T^{n-1}\omega)\cdots A(T\omega)A(\omega),\,\, \mathrm{for}\,\ n\ge1,
\end{equation*}
and $A_{-n}(\omega)=A_{n}(T^{-n}\omega)^{-1}$.

We denote by $L_1(A)\geq L_2(A)\geq...\geq L_m(A)$ the Lyapunov exponents of $(T,A)$ repeated according to their multiplicities, i.e.,
$$
L_k(A)=\lim\limits_{n\rightarrow\infty}\frac{1}{n}\int_{\Omega}\ln(\sigma_k(A_n(\omega)))d \mathbb P(\omega),
$$
where for any matrix $B\in {\rm M_m}(\C)$, we denote by
$\sigma_1(B)\geq...\geq \sigma_m(B)$ its singular values (eigenvalues
of $\sqrt{B^*B}$).  Since the $k$-th exterior product $\Lambda^k A_n$ satisfies $\sigma_1(\Lambda^k A_n)=\|\Lambda^k A_n\|$, $L^k(A)=\sum_{j=1}^kL_j(A)$ satisfies
$$
L^k(A)=\lim\limits_{n\rightarrow \infty}\frac{1}{n}\int_{\Omega}\ln\|\Lambda^kA_n(\omega)\|d\mathbb P(\omega).
$$

\subsection{Density of states}
 Let $a,b\in\mathbb Z$ with $a\leq b$ and set
$\Lambda_{a,b}:=\{a,a+1,\ldots,b\}$. 
Let
\begin{equation*}
P_{a,b}:
\ell^2(\mathbb Z,\mathbb C^m)
\longrightarrow
\ell^2(\Lambda_{a,b},\mathbb C^m)
\end{equation*}
be the coordinate restriction, and define the finite-volume Dirichlet restriction by
$\mathcal H_{\omega,[a,b]}:=P_{a,b}\mathcal H_\omega P_{a,b}^*.$
Let
\begin{equation*}
\lambda_1^{[a,b]}(\omega),\ldots,\lambda_{m|\Lambda_{a,b}|}^{[a,b]}(\omega)
\end{equation*}
be the eigenvalues of $\mathcal H_{\omega,[a,b]}$, repeated according to their algebraic multiplicities. The normalized finite-volume eigenvalue counting measure is defined by
\begin{equation}\label{equ:IDS}
\mathcal N_{\omega,[a,b]}
:=
\frac{1}{m|\Lambda_{a,b}|}
\sum_{j=1}^{m|\Lambda_{a,b}|}
\delta_{\lambda_j^{[a,b]}(\omega)}.
\end{equation}
By the standard arguments \cite{Damanik-Fillman-book,Hof}, $\mathcal N_{\omega,[a,b]}\rightharpoonup \mathcal N$  for $\mathbb P$-almost every $\omega$.
Moreover, if $(\Omega,\mathbb P,T)$ is uniquely ergodic and $A(\cdot),B(\cdot)$ are continuous, $\mathcal N_{\omega,[a,b]}\rightharpoonup \mathcal N$ for every $\omega$.
And we have the following Thouless formula:
\begin{theorem}[Thouless formula]\cite{CS,Puig,KS}\label{thouless}
     Let $L_1(z)\geq L_2(z)\geq\cdots\geq L_m(z)$
denote the nonnegative Lyapunov exponents. Then, for every $z\in\mathbb C$,
\begin{equation*}
\frac{1}{m}\sum_{j=1}^mL_j(z)=-\frac{1}{m}\int_\Omega\log|\det A(\omega)|
d\mathbb P(\omega)+
\int_{\mathbb R}\log|z-E|
d\mathcal N(E).
\end{equation*}
\end{theorem}

\subsection{Potential theory}
We recall some basic notions from logarithmic potential theory.

\begin{definition}\label{cap}
    Let $\rho$ be a finite Borel measure on $\C$ with compact support. Its energy $I(\rho)$ is given by
\begin{equation*}\label{eq:energy}
I(\rho):=\int\int\log{|z-w|}d\rho(z)d\rho(w).
\end{equation*}
The  capacity of a subset $E$ of $\C$ is given by $\operatorname{cap}(E):=\sup_\rho e^{I(\rho)}$, where the supremum is taken over all Borel probability measure $\rho$ on $\C$ whose support is a compact subset of $E$.
\end{definition}

\begin{theorem}[Upper envelope theorem]\cite{Landkof}
\label{thm:upper-envelope}
Let $K\subset\mathbb C$ be compact, and let $\{\rho_n\}_{n\geq1}$ and $\rho$ be finite Borel measures supported on $K$. Suppose that  $\rho_n\rightharpoonup\rho$
weakly. Then, for every $z\in\mathbb C$,
\begin{equation*}
\limsup_{n\to\infty}\int\log |z-w|d\rho_n(w)\leq\int\log |z-w|d\rho(w).
\end{equation*}
Furthermore, there exists a set $\mathcal E\subset\mathbb C$ with
$\operatorname{cap}(\mathcal E)=0$ such that, for every
$z\in\mathbb C\setminus\mathcal E$,
\begin{equation*}
\limsup_{n\to\infty}\int_{\mathbb C}\log|z-w|\,d\rho_n(w)=\int_{\mathbb C}\log|z-w|\,d\rho(w).
\end{equation*}
\end{theorem}

\section{Unitary flipping}\label{flip}

The construction below is inspired by González and Sadel
\cite[Remarks (g)]{GS}, who used the same geometric observation for
whole-line random Schr\"odinger operators with constant hopping. Here we
implement it for general block Jacobi operators and record the
finite-volume and spectral-measure identifications needed below. 

\begin{figure}[ht]
\centering
\begin{tikzpicture}[
  x=1.15cm,
  y=1.0cm,
  site/.style={
    circle,
    draw,
    thick,
    minimum size=7mm,
    inner sep=0pt
  },
  psite/.style={
    site,
    draw=blue!70!black,
    fill=blue!8
  },
  nsite/.style={
    site,
    draw=orange!80!black,
    fill=orange!12
  },
  bond/.style={thick},
  foldarrow/.style={
    -{Latex[length=2.5mm]},
    thick,
    red!70!black
  }
]

\node at (0,2.05) {\textbf{Whole-line}};

\node at (-4,1.25) {$\cdots$};
\node[nsite] (m3) at (-3,1.25) {$-3$};
\node[nsite] (m2) at (-2,1.25) {$-2$};
\node[nsite] (m1) at (-1,1.25) {$-1$};

\node[psite] (p0) at (0,1.25) {$0$};
\node[psite] (p1) at (1,1.25) {$1$};
\node[psite] (p2) at (2,1.25) {$2$};
\node at (3,1.25) {$\cdots$};

\draw[bond,orange!80!black]
  (-3.65,1.25)--(m3)--(m2)--(m1);

\draw[very thick,red!70!black]
  (m1)--(p0);

\draw[bond,blue!70!black]
  (p0)--(p1)--(p2)--(2.65,1.25);

\draw[dashed]
  (-0.5,0.82)--(-0.5,1.68);

\draw[foldarrow]
  (-2.85,0.9)
  .. controls (-2.75,0.15) and (-1.7,-0.05) ..
  (-1.1,-0.35);

\node[red!70!black] at (-2.55,0.15)
  {$\text{flip}$};

\node at (0,0.25)
  {\textbf{Half-line }};

\node[psite] (t0) at (-0.5,-0.55) {$0$};
\node[psite] (t1) at (1.0,-0.55) {$1$};
\node[psite] (t2) at (2.5,-0.55) {$2$};
\node at (3.7,-0.55) {$\cdots$};

\node[nsite] (b0) at (-0.5,-1.55) {$-1$};
\node[nsite] (b1) at (1.0,-1.55) {$-2$};
\node[nsite] (b2) at (2.5,-1.55) {$-3$};
\node at (3.7,-1.55) {$\cdots$};

\draw[bond,blue!70!black]
  (t0)--(t1)--(t2)--(3.35,-0.55);

\draw[bond,orange!80!black]
  (b0)--(b1)--(b2)--(3.35,-1.55);

\draw[very thick,red!70!black]
  (t0)--(b0);

\draw[rounded corners,thick]
  (-0.95,-0.13) rectangle (-0.05,-1.97);

\draw[rounded corners,thick]
  (0.55,-0.13) rectangle (1.45,-1.97);

\draw[rounded corners,thick]
  (2.05,-0.13) rectangle (2.95,-1.97);

\node at (-0.5,-2.3) {$n=0$};
\node at (1.0,-2.3) {$n=1$};
\node at (2.5,-2.3) {$n=2$};

\end{tikzpicture}

\caption{Flipping transformation}
\label{fig:folding}
\end{figure}
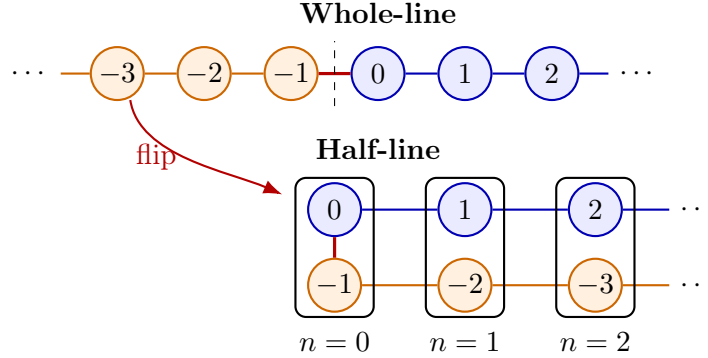

As illustrated in Figure~\ref{fig:folding}, we flip the negative
half-line onto the nonnegative half-line about the bond
$\{-1,0\}$. More precisely, define
$U:\ell^2(\mathbb Z,\mathbb C^m)
\longrightarrow
\ell^2(\mathbb Z_+,\mathbb C^{2m})$
by
\begin{equation*}
(Uu)_n
=
\binom{u_n}{u_{-n-1}},
\qquad n\geq0.
\end{equation*}
Thus, the $n$-th block of the flipped half-line consists of the two
sites $n$ and $-n-1$ of the original whole-line operator. 
The coefficients of the flipped operator can now be read directly from
this identification. For $n\geq0$, set
\begin{equation*}
\mathcal A_n(\omega)
=
\begin{pmatrix}
A_n(\omega)&0\\
0&A_{-n-2}(\omega)^*
\end{pmatrix}.
\end{equation*}
Away from the zeroth block, the two channels are uncoupled, and hence
\begin{equation*}
\mathcal B_n(\omega)
=
\begin{pmatrix}
B_n(\omega)&0\\
0&B_{-n-1}(\omega)
\end{pmatrix},
\qquad n\geq1.
\end{equation*}
At the zeroth block, the original hopping across the bond
$\{-1,0\}$ becomes an intra-block coupling. Therefore,
\begin{equation*}
\mathcal B_0(\omega)
=
\begin{pmatrix}
B_0(\omega)&A_{-1}(\omega)^*\\
A_{-1}(\omega)&B_{-1}(\omega)
\end{pmatrix}.
\end{equation*}

We define the flipped operator
$\mathcal J_\omega$ on
$\ell^2(\mathbb Z_+,\mathbb C^{2m})$ by
\begin{equation*}
(\mathcal J_\omega y)_n=\mathcal A_{n-1}(\omega)^*y_{n-1}
+\mathcal B_n(\omega)y_n
+\mathcal A_n(\omega)y_{n+1},
\qquad n\geq0,
\end{equation*}
with the Dirichlet boundary condition $y_{-1}=0$.
The operator $\mathcal J_\omega$ is thus a half-line block Jacobi
operator of block size $2m$. Notice that, although
$\mathcal H_\omega$ belongs to an ergodic family, the flipped family
$\{\mathcal J_\omega\}_{\omega\in\Omega}$ is generally not ergodic.

We shall also need the corresponding finite-volume identification.
For $N\geq1$, let
$\mathcal H_{\omega,[-N,N-1]}$ denote the Dirichlet restriction of
$\mathcal H_\omega$ to $[-N,N-1]$, with boundary conditions
$u_{-N-1}=u_N=0,$
and let $\mathcal J_{\omega,[0,N-1]}$ denote the Dirichlet
restriction of $\mathcal J_\omega$ to $[0,N-1]$, with boundary
conditions 
$y_{-1}=y_N=0.$
Define
$$
U^{(N)}:
\ell^2([-N,N-1],\mathbb C^m)
\longrightarrow
\ell^2([0,N-1],\mathbb C^{2m})
$$
by
\begin{equation*}
(U^{(N)}u)_k=
\binom{u_k}{u_{-k-1}},
\qquad 0\leq k\leq N-1.
\end{equation*}
The maps $U$ and $U^{(N)}$ provide the infinite- and
finite-volume unitary equivalences recorded in the following lemma.

\begin{lemma}\label{lem:unitary}
The flipping map $U$ is unitary and 
$\mathcal J_\omega
=
U\mathcal H_\omega U^{-1}.$
Moreover, for every $N\geq1$, $U^{(N)}$ is unitary and
$$
\mathcal J_{\omega,[0,N-1]}
=
U^{(N)}
\mathcal H_{\omega,[-N,N-1]}
\bigl(U^{(N)}\bigr)^{-1}.
$$
\end{lemma}

\begin{proof}
Both maps merely rearrange the coordinates, which is easily check to be unitary. 
If $y=Uu$, then the definitions above give
$$
(\mathcal J_\omega y)_n
=
\binom{(\mathcal H_\omega u)_n}
      {(\mathcal H_\omega u)_{-n-1}},
\qquad n\geq0.
$$
Hence $\mathcal J_\omega U=U\mathcal H_\omega$. The finite-volume
identity follows from the same computation, since
$u_{-N-1}=u_N=0$ corresponds exactly to $y_{-1}=y_N=0$.
\end{proof}

The unitary equivalence above also identifies the natural spectral
measure of $\mathcal H_\omega$ associated with the two adjacent sites
$-1$ and $0$ with the trace spectral measure of
$\mathcal J_\omega$ at its zeroth block.

Let
$\bigl\{\delta_{i,n}:1\leq i\leq m,\ n\in\mathbb Z\bigr\}$
and $
\bigl\{\widetilde\delta_{j,n}:1\leq j\leq 2m,\
n\in\mathbb Z_+\bigr\}$
denote the standard orthonormal bases of
$\ell^2(\mathbb Z,\mathbb C^m)$ and
$\ell^2(\mathbb Z_+,\mathbb C^{2m})$, respectively. For every Borel
set $\Delta\subset\mathbb R$, define
\begin{equation*}
\mu_\omega(\Delta):=\sum_{i=1}^m\left\langle\delta_{i,0},\chi_\Delta(\mathcal H_\omega)\delta_{i,0}\right\rangle +
\sum_{i=1}^m\left\langle\delta_{i,-1},\chi_\Delta(\mathcal H_\omega)\delta_{i,-1}\right\rangle
\end{equation*}
and
\begin{equation*}
\tau_\omega(\Delta):=\sum_{j=1}^{2m}\left\langle\widetilde\delta_{j,0},\chi_\Delta(\mathcal J_\omega)\widetilde\delta_{j,0}
\right\rangle .
\end{equation*}
Since the hopping matrices are invertible, the vectors supported at
the sites $-1$ and $0$ form a cyclic family for
$\mathcal H_\omega$, while the vectors supported at the zeroth block
form a cyclic family for $\mathcal J_\omega$. Consequently,
$\mu_\omega$ and $\tau_\omega$ are maximal spectral measures of
$\mathcal H_\omega$ and $\mathcal J_\omega$, respectively. As a consequence of unitary equivalance, we have the following: 

\begin{lemma}\label{lem:measure}
For every Borel set $\Delta\subset\mathbb R$, 
$\mu_\omega(\Delta)=\tau_\omega(\Delta).$
\end{lemma}

\begin{proof}
By Lemma~\ref{lem:unitary} and the functional calculus, 
$\chi_\Delta(\mathcal J_\omega)
=
U\chi_\Delta(\mathcal H_\omega)U^{-1}.$
Moreover, the definition of $U$ gives $
U\delta_{i,0}=\widetilde\delta_{i,0},$ $U\delta_{i,-1}=\widetilde\delta_{m+i,0},\, 1\leq i\leq m.$
Therefore,
$$
\mu_\omega(\Delta)=\sum_{i=1}^m\left\langle\widetilde\delta_{i,0},
\chi_\Delta(\mathcal J_\omega)\widetilde\delta_{i,0}\right\rangle +
\sum_{i=1}^m\left\langle\widetilde\delta_{m+i,0},\chi_\Delta(\mathcal J_\omega)\widetilde\delta_{m+i,0}\right\rangle
=\tau_\omega(\Delta).
$$
\end{proof}

\section{Growth of the flipped Dirichlet fundamental matrix}\label{sec:Dirichlet}

\subsection{Transfer matrices and Dirichlet solutions}

In the preceding section, the flipping map identifies the whole-line operator
$\mathcal H_\omega$ with the half-line block Jacobi operator $\mathcal J_\omega$ on
$\ell^2(\mathbb Z_+,\mathbb C^{2m})$. 

We first introduce the transfer matrix associated with the eigenvalue
equation
$\mathcal J_\omega y=Ey.$
For $n\geq0$, define the one-step transfer matrix
\begin{equation*}
\mathscr M_n(E,\omega):=
\begin{pmatrix}
\mathcal A_n(\omega)^{-1}\bigl(EI_{2m}-\mathcal B_n(\omega)\bigr)
&-\mathcal A_n(\omega)^{-1}\mathcal A_{n-1}(\omega)^*\\
I_{2m}&0
\end{pmatrix},
\end{equation*}
with the convention $\mathcal A_{-1}=I_{2m}$. Note that $\mathcal A_n(\omega)$  depends on both $T^n \omega$ and  $T^{-n-2}\omega$, so the sequence $\mathscr M_n(E,\omega)$   is not generated by an ergodic cocycle over $(\Omega,\mathbb P,T)$.
For $n\geq1$, define the half-line transfer matrix by
\begin{equation*}
\mathscr T(n;E,\omega):=\mathscr M_{n-1}(E,\omega)\cdots
\mathscr M_1(E,\omega)
\mathscr M_0(E,\omega).
\end{equation*}
Let $D_n(E,\omega)$ be the Dirichlet fundamental matrix associated
with $\mathcal J_\omega$, determined by
$
D_{-1}(E,\omega)=0,
D_0(E,\omega)=I_{2m},$
and
$$
\mathcal A_n(\omega)D_{n+1}(E,\omega)
=
\bigl(EI_{2m}-\mathcal B_n(\omega)\bigr)D_n(E,\omega)
-\mathcal A_{n-1}(\omega)^*D_{n-1}(E,\omega).
$$
With $\mathcal A_{-1}=I_{2m}$, this is equivalently
expressed as
\begin{equation*}
\binom{D_n(E,\omega)}{D_{n-1}(E,\omega)}
=
\mathscr T(n;E,\omega)
\binom{I_{2m}}{0},
\qquad
n\geq1.
\end{equation*}
Therefore, $D_n(E,\omega)$ is precisely the upper-left block of the
$n$-step transfer matrix of the half-line operator.

Although $D_n(E,\omega)$ is the upper-left block of the half-line transfer
matrix, the sequence $\{\mathscr M_n(E,\omega)\}_{n\geq0}$ is not generated
by an ergodic cocycle over the original system $(\Omega,\mathbb P,T)$. 
The absence of an underlying ergodic cocycle prevents
a direct application of the standard ergodic theorems to its asymptotic
growth. However, its finite-volume spectral statistics are still governed by
the density of states of the original whole-line operator. We record this
observation together with the Thouless formula.

\begin{proposition}\label{lem:IDS}
Let $\lambda_{n,\omega,1},\ldots,\lambda_{n,\omega,2mn}$
be the eigenvalues of $\mathcal J_{\omega,[0,n-1]}$ and define the normalized eigenvalue counting measure by
\begin{equation*}
\mathcal N_{n,\omega}^{(\mathcal J)}:=\frac{1}{2mn}\sum_{j=1}^{2mn}\delta_{\lambda_{n,\omega,j}}.
\end{equation*}
Let $\mathcal{N}$ be the density of states measure of the whole-line operator $\mathcal H_\omega$.
Then, for $\mathbb P$-a.e. $\omega$, $
\mathcal N_{n,\omega}^{(\mathcal J)}
\rightharpoonup
\mathcal N$
weakly as $n\to\infty$.

In addition, if $(\Omega,\mathbb P,T)$ is uniquely ergodic and
$A(\cdot)$ and $B(\cdot)$ are continuous, then
$
\mathcal N_{n,\omega}^{(\mathcal J)}
\rightharpoonup
\mathcal N$ holds for every $\omega\in\Omega$.
\end{proposition}
\begin{proof}
By Lemma~\ref{lem:unitary}, $
\mathcal J_{\omega,[0,n-1]}
=
U^{(n)}
\mathcal H_{\omega,[-n,n-1]}
\bigl(U^{(n)}\bigr)^{-1}.$
Hence the two finite-volume restrictions have the same eigenvalues,
counted with multiplicity, and therefore
$\mathcal N_{n,\omega}^{(\mathcal J)}
=
\mathcal N_{\omega,[-n,n-1]},$ where $\mathcal N_{\omega,[-n,n-1]}$ is defined in \eqref{equ:IDS}.
The asserted convergence now follows from the standard finite-volume
approximation of the density of states
\cite{Damanik-Fillman-book,Hof}. In the uniquely ergodic
setting, the corresponding convergence holds for every
$\omega\in\Omega$.
\end{proof}

\subsection{Determinant asymptotics}

Proposition \ref{lem:IDS} shows that, despite the loss of ergodicity after the
flip, the finite-volume eigenvalue counting measures of the half-line
operator converge to the density of states of the original whole-line
family. This allows us to control the determinant of the Dirichlet
fundamental matrix through logarithmic potential theory and the Thouless
formula.

\begin{lemma}\label{lem:1}
For $\mathbb P$-a.e. $\omega$, there exists a Borel set
$Q_{\omega}\subset\mathbb R$ with capacity zero such that for any $E\notin Q_{\omega}$, 
\begin{equation*}\label{singular-value-(3)}
\limsup_{n\to\infty}
\frac1n\log|\det D_n(E,\omega)|=
2\sum_{j=1}^m L_j(E).
\end{equation*}
Moreover, if $(\Omega,\mathbb P,T)$ is uniquely ergodic and $A(\cdot),B(\cdot)$ are continuous, then the above conclusion holds for every $\omega$. 
\end{lemma}
\begin{proof}
By the definition,
$$
\frac1n\log|\det(E I-\mathcal J_{\omega,[0,n-1]})|=2m\int_{\mathbb R}\log|E-x|d\mathcal N_{n,\omega}^{(\mathcal J)}(x).
$$
By Proposition \ref{lem:IDS}, for $\mathbb P$-a.e. $\omega$, $
\mathcal N_{n,\omega}^{(\mathcal J)}
\rightharpoonup
\mathcal N$
weakly as $n\to\infty$.
By the upper envelope theorem (Theorem \ref{thm:upper-envelope}) for logarithmic potentials, for $\mathbb P$-a.e. $\omega$,
there exists a set $Q_{\omega}$ of  capacity zero, such that
for every $E\notin Q_{\omega}$,
$$\limsup_{n\to\infty}\int_{\mathbb R}\log|E-x|d\mathcal N_{n,\omega}^{(\mathcal J)}(x)=\int_{\mathbb{R}}\log|E-x|d\mathcal N(x).$$
Thus, for $\mathbb P$-a.e. $\omega$,
$$
\limsup_{n\to\infty}
\frac1n\log|\det(E-\mathcal J_{\omega,[0,n-1]})|=2m\int_{\mathbb{R}}\log|E-x|d\mathcal N(x)
$$
for every  $E\notin Q_{\omega}$. In the uniquely ergodic
setting, the above equation holds for every $\omega$ with  $E\notin Q_{\omega}$.

Applying  Theorem \ref{thouless}, we have
\begin{equation}\label{singular-value-(2)}
\sum_{j=1}^m L_j(E)=m\int_{\mathbb R}\log|E-x|d \mathcal N(x)-\int_{\Omega}\log|\det A(\omega)|\,d\mathbb P(\omega).
\end{equation}

On the other hand, using the standard calculation \cite{CS}, we get
$$
\det(E-\mathcal J_{\omega,[0,n-1]})=\left(
\prod_{j=0}^{n-1}\det\mathcal A_j(\omega)\right)
\det D_n(E,\omega).
$$
Therefore,
\begin{equation}\label{singular-value-(1)}
\begin{aligned}
&\quad\frac1n\log|\det D_n(E,\omega)|=\frac1n\log|\det(E-\mathcal J_{\omega,[0,n-1]})|-
\frac1n\sum_{j=0}^{n-1}\log|\det\mathcal A_j(\omega)|.
\end{aligned}
\end{equation}
By the definition of the flipped hopping matrices,
$$
\log|\det\mathcal A_j(\omega)|
=
\log|\det A(T^j\omega)|
+
\log|\det A(T^{-j-2}\omega)|.
$$
Since $T$ is invertible and ergodic, the forward and backward
versions of Birkhoff's theorem imply that, for
$\mathbb P$-almost every $\omega$,
\begin{equation}\label{equ:ergodic}
\lim_{n\rightarrow \infty}\frac1n\sum_{j=0}^{n-1}\log|\det\mathcal A_j(\omega)|=
2\int_{\Omega}\log|\det A(\omega)|d\mathbb P(\omega).
\end{equation}
Moreover, if $(\Omega,\mathbb P,T)$ is uniquely ergodic and $A(\cdot)$ are continuous and invertible, by uniquely ergodic theorem, equation \eqref{equ:ergodic} holds for every $\omega.$
Combining with equations \eqref{singular-value-(2)}, \eqref{singular-value-(1)}, and \eqref{equ:ergodic}, we have the conclusion.

\end{proof}

\subsection{The smallest singular value}

Let $s_{k}(D_n(E,\omega))$ be the $k$-th singular value of $D_n(E,\omega)$, and $s_{1}\geq s_2 \geq \cdots \geq s_{2m}$. We have the following estimate.
\begin{proposition}\label{lem:singular}
For $\mathbb P$-a.e. $\omega$, there exists a Borel set
$Q_{\omega}\subset\mathbb R$ with capacity zero such that for
every $E\notin Q_{\omega}$,
$$
\limsup_{n\to\infty}\frac1n\log s_{2m}(D_n(E,\omega))\ge L_m(E).
$$Moreover, if $(\Omega,\mathbb P,T)$ is uniquely ergodic and $A(\cdot),B(\cdot)$ are continuous, then the above conclusion holds for every $\omega$.
\end{proposition}
\begin{proof}
Since
$|\det D_n(E,\omega)|=
s_{2m}(D_n(E,\omega))\|\wedge^{2m-1}D_n(E,\omega)\|,
$ the proof of Proposition \ref{lem:singular} reduces to the following lemma:

\begin{lemma}\label{lem:2}
For $\mathbb P$-a.e. $\omega$,
$$
\limsup_{n\to\infty}
\frac1n\log\|\wedge^{2m-1}D_n(E,\omega)\|
\le2\sum_{j=1}^mL_j(E)-L_m(E).
$$    
Moreover, if $(\Omega,\mathbb P,T)$ is uniquely ergodic and $A(\cdot),B(\cdot)$ are continuous, then the above conclusion holds for every $\omega$. 
\end{lemma}
\begin{proof}
We first give the upper estimate of $\wedge^qD_n(E,\omega)$ in terms of the original transfer matrix.
The forward one-step transfer matrix is defined by
\begin{equation*}
T^+(E,\omega):=
\begin{pmatrix}
A(\omega)^{-1}(EI_m-B(\omega))
&-A(\omega)^{-1}A(T^{-1}\omega)^*\\
I_m&0
\end{pmatrix}.
\end{equation*}
For $n\geq1$, define
\begin{equation*}
\mathcal T^+(n;E,\omega):=
T^+(E,T^{n-1}\omega)\cdots
T^+(E,T\omega)
T^+(E,\omega).
\end{equation*}
Thus,
\begin{equation}\label{eq:forwardp}
\binom{u_n}{u_{n-1}}=
\mathcal T^+(n;E,\omega)
\binom{u_0}{u_{-1}}.
\end{equation}
Similarly, the backward one-step transfer matrix is defined by
\begin{equation*}
T^-(E,\omega):=
\begin{pmatrix}
(A(T^{-1}\omega)^*)^{-1}(EI_m-B(\omega))&
-(A(T^{-1}\omega)^*)^{-1}A(\omega)\\
I_m&0
\end{pmatrix}.
\end{equation*}
For $n\geq1$, define
\begin{equation*}
\mathcal T^-(n;E,\omega):=
T^-(E,T^{-n}\omega)\cdots T^-(E,T^{-2}\omega)T^-(E,T^{-1}\omega).
\end{equation*}
It follows that
\begin{equation}\label{eq:backwardp}
\binom{u_{-n-1}}{u_{-n}}=
\mathcal T^-(n;E,\omega)
\binom{u_{-1}}{u_0}.
\end{equation}

\begin{lemma}
\label{lem:flipped}
For every $0\le q\le 2m$, there exists a constant $C_{m,q}>0$, independent of
$n,E,\omega$, such that
$$
\|\wedge^qD_n(E,\omega)\|
\le C_{m,q}\max_{\max\{0,q-m\}\leq a\leq\min\{q,m\}}
\left\{\|\wedge^a\mathcal T^+(n;E,\omega)\|
\|\wedge^{q-a}\mathcal T^-(n;E,\omega)\|
\right\}.
$$
\end{lemma}

\begin{proof}

For every
$v=\binom{x}{y}\in\mathbb C^{2m},$
let $u=\{u_j\}_{j\in\mathbb Z}$ be the unique solution of
$\mathcal H_\omega u=Eu$
with initial data
$u_0=x,u_{-1}=y.$
Then
\begin{equation}\label{eq:d}
D_n(E,\omega)\binom{x}{y}
=\binom{u_n}{u_{-n-1}}.
\end{equation}

 By equations \eqref{eq:forwardp} and \eqref{eq:backwardp}, 
\begin{equation*}
\binom{u_n}{u_{n-1}}=
\mathcal T^+(n;E,\omega)
\binom{u_0}{u_{-1}},\qquad
\binom{u_{-n-1}}{u_{-n}}=
\mathcal T^-(n;E,\omega)
\binom{u_{-1}}{u_0}.
\end{equation*}
Write
$$
\mathcal T^+(n;E,\omega)=
\begin{pmatrix}
T_{11}^+(n)&T_{12}^+(n)\\
T_{21}^+(n)&T_{22}^+(n)
\end{pmatrix},\qquad
\mathcal T^-(n;E,\omega)=
\begin{pmatrix}
T_{11}^-(n)&T_{12}^-(n)\\
T_{21}^-(n)&T_{22}^-(n)
\end{pmatrix}.
$$
Then

\begin{equation}\label{eq:transfer}
u_n=
T_{11}^+(n)x+T_{12}^+(n)y, \qquad u_{-n-1}=
T_{12}^-(n)x+T_{11}^-(n)y.
\end{equation}

Combining equations \eqref{eq:d} and
\eqref{eq:transfer}, we find
$$D_n(E,\omega)\binom{x}{y}=
\binom{u_n}{u_{-n-1}}=
\begin{pmatrix}
T_{11}^+(n)&T_{12}^+(n)\\
T_{12}^-(n)&T_{11}^-(n)
\end{pmatrix}
\binom{x}{y}.
$$
Since this identity holds for every $\binom{x}{y}\in\mathbb C^{2m}$, we conclude that
\begin{equation*}\label{eq:flipped}
D_n(E,\omega)=\begin{pmatrix}
T_{11}^+(n)&T_{12}^+(n)\\
T_{12}^-(n)&T_{11}^-(n)
\end{pmatrix}.
\end{equation*}

Let
$$
A=(T_{11}^+(n)\,\,\ T_{12}^+(n)),
\quad
B=(T_{12}^-(n)\,\,\ T_{11}^-(n)).
$$
For a matrix $M$, 
$M[I,J]$ denotes the minor with row set $I$ and column set
$J$.
Fix $I,J\subset\{1,\ldots,2m\}$ with $|I|=|J|=q$. 
Let
$$I_+:=I\cap\{1,\ldots,m\},
I_-:=\{i-m:i\in I\cap\{m+1,\ldots,2m\}\}.
$$
Set
$a:=|I_+|.$
Then
$
|I_-|=q-a,
$
and necessarily
$$
\max\{0,q-m\}\leq a\leq\min\{q,m\}.
$$
The corresponding $q\times q$ minor
$\det D_n[I,J]$ has the form
$$
\det
\begin{pmatrix}
A[I_+,J]\\
B[I_-,J]
\end{pmatrix}.
$$
By the Laplace expansion with respect to the first $a$ rows,
$$
\det
\begin{pmatrix}
A[I_+,J]\\
B[I_-,J]
\end{pmatrix}
=\sum_{\substack{J_+\subset J\\ |J_+|=a}}\pm
\det A[I_+,J_+]\,
\det B[I_-,J_-],
$$
where $J_{-}=J\setminus J_+$.
Hence
$$
|\det D_n[I,J]|
\le
\sum_{\substack{J_+\subset J\\ |J_+|=a}}
|\det A[I_+,J_+]|\,
|\det B[I_-,J_-]|.
$$
Since $A[I_+,J_+]$ is an $a\times a$ minor of
$\mathcal T^+(n;E,\omega)$, we have
$$
|\det A[I_+,J_+]|\le
\|\wedge^a\mathcal T^+(n;E,\omega)\|.
$$
Similarly, since $B=( T_-^{12}(n)\ \  T_-^{11}(n))$ consists of the first $m$ rows of $\mathcal T_-(n;E,\omega)$ up to a permutation of the two column blocks, each submatrix $B[I_-,J_-]$ coincides, up to sign, with a $(q-a)\times(q-a)$ minor of $\mathcal T_-(n;E,\omega)$. Consequently,
$$
|\det B[I_-,J_-]|\le
\|\wedge^{q-a}\mathcal T_-(n;E,\omega)\|.
$$
Thus
$$
|\det D_n[I,J]|
\le
\binom{q}{a}
\|\wedge^a\mathcal T^+(n;E,\omega)\|
\|\wedge^{q-a}\mathcal T^-(n;E,\omega)\|.
$$Therefore,
$$
\|\wedge^qD_n(E,\omega)\|
\le C_{m,q}\max_{\max\{0,q-m\}\leq a\leq\min\{q,m\}}
\left\{\|\wedge^a\mathcal T^+(n;E,\omega)\|
\|\wedge^{q-a}\mathcal T^-(n;E,\omega)\|
\right\}.
$$
\end{proof}

Once we have this statement, we can now proceed with the proof of Lemma \ref{lem:2}.
    Let $q=2m-1$. It follows that
\begin{equation*}
\|\wedge^{2m-1}D_n(E,\omega)\|\leq C_m
\max_{a\in\{m-1,m\}}\left\{
\|\wedge^a\mathcal T^+(n;E,\omega)\|
\|\wedge^{2m-1-a}\mathcal T^-(n;E,\omega)\|
\right\}.
\end{equation*}
We obtain
\begin{align*}
\limsup_{n\to\infty}
\frac1n\log\|\wedge^{2m-1}D_n(E,\omega)\|
&\leq\max_{a\in\{m-1,m\}}\Bigg\{\limsup_{n\to\infty}
\frac1n\log\|\wedge^a\mathcal T^+(n;E,\omega)\|\\
&+\limsup_{n\to\infty}\frac1n
\log\|\wedge^{2m-1-a}\mathcal T^-(n;E,\omega)\|
\Bigg\}.
\end{align*}
A direct computation shows that
\begin{equation*}
T^+(E,\omega)^* J(T\omega)T^+(E,\omega)=J(\omega).
\end{equation*}
where \begin{equation*} J(\omega)=
\begin{pmatrix}
0&-A(T^{-1}\omega)^*\\
A(T^{-1}\omega)&0
\end{pmatrix}.
\end{equation*}
Thus the cocycle $(T,T^+(E,\cdot))$ is symplectic with respect to the
structure $J(\omega)$. Therefore, its Lyapunov
exponents satisfies:
$L_j(E)=-L_{2m+1-j}(E)$ for $j=1,\dots,m$.
Moreover, since  $\mathcal R:=
\begin{pmatrix}
0&I_m\\
I_m&0
\end{pmatrix}
$
and $\mathcal R^{-1}=\mathcal R$, a direct calculation gives
\begin{equation*}
T^-(E,\omega)=
\mathcal R T^+(E,\omega)^{-1} \mathcal R.
\end{equation*}
Since Lyapunov exponents are invariant under conjugacy, the cocycles $(T,T^+(E,\cdot))$ and $(T,T^-(E,\cdot))$ share the same Lyapunov exponents $\{L_i(E)\}_{i=1}^{2m}$
due to  $L_j(E)=-L_{2m+1-j}(E)$.

As $(T,T^+(E,\cdot))$ is ergodic, we recall the Craig--Simon  type upper bound in the form needed below.
\begin{lemma}\label{lem:Craig-Simon}\cite{CS83-Duke,GS24-IMRN}
For $\mathbb P$-almost every $\omega$, every $E\in\mathbb R$, and every
$1\leq i\leq2m$,
\begin{equation}\label{eq:Craig-Simon}
\limsup_{n\to\infty}\frac{1}{n}\log\|\wedge^i\mathcal T^\pm(n;E,\omega)\|
\leq L^i(E),
\end{equation}
where
$L^i(E):=\sum_{k=1}^{i}L_k(E).$
\end{lemma}
Applying Lemma \ref{lem:Craig-Simon}, one can obtain that 
\begin{align*}
\limsup_{n\to\infty}
\frac1n\log\|\wedge^{2m-1}D_n(E,\omega)\|
&\leq\max_{a\in\{m-1,m\}}\left\{
L^a(E)+L^{2m-1-a}(E)
\right\}\\
&=L^{m-1}(E)+L^m(E)=
2\sum_{j=1}^mL_j(E)-L_m(E).
\end{align*}
If $(\Omega,\mathbb P,T)$ is uniquely ergodic and $A(\cdot),B(\cdot)$
are continuous, then recall famous result of Furman \cite{furman}, \eqref{eq:Craig-Simon} holds for every $\omega\in\Omega$.

\end{proof}
By the fact that $\limsup(a_n-b_n)\geq \limsup a_n -\limsup b_n$,  Proposition \ref{lem:singular} follows immediately from Lemma \ref{lem:1} and Lemma \ref{lem:2}.
\end{proof}

\subsection{Proof of main results}

\begin{proof}[Proof of Theorem \ref{thm:main}]
Fix $\omega$ for which the Proposition \ref{lem:singular} holds.
Suppose that
$
E\in S_{+}\setminus Q_{\omega}.
$
Then $L_m(E)>0$, by Proposition \ref{lem:singular}, we have
$$
\limsup_{n\to\infty}
\frac1n\log s_{2m}(D_n(E,\omega))
\ge
L_m(E)>0.
$$
Therefore, there exist a subsequence $n_j\to\infty$ and a number $\eta(E)>0$
such that
$$
s_{2m}(D_{n_j}(E,\omega))\ge e^{\eta(E)n_j}.
$$
On the other hand, by the Shnol theorem \cite{Lenz,Poerschke,Schlag}, for $\tau_\omega$-almost every
$E$, there exists a nonzero polynomially bounded solution of $
\mathcal J_\omega\psi=E\psi.$ In other words,
for $\tau_{\omega}$-almost every $E$, there exists  a nonzero vector ${v}_{0}$ satisfying
$
\|D_n(E,\omega){v}_{0}\|\le C(1+n).
$
These are incompatible unless
$
\tau_{\omega}(S_{+}\setminus Q_{\omega})=0.
$
By Lemma \ref{lem:measure}, we have that
$
\mu_\omega(S_{+}\setminus Q_\omega)=
\tau_{\omega}(S_{+}\setminus Q_\omega)=0.
$
\end{proof}

\begin{proof}[Proof of Corollary \ref{cor:uniform}]
 Proposition \ref{lem:singular} holds for every $\omega\in\Omega$ when
$(\Omega,\mathbb P,T)$ is uniquely ergodic and $A(\cdot),B(\cdot)$ are
continuous, the proof is the same as that of Theorem \ref{thm:main}.
\end{proof}

\begin{proof}[Proof of Theorem \ref{thm:scalar}]
It is an immediate corollary of Theorem \ref{thm:main} in the case $m=1$ and $A(\cdot)\equiv 1$.
\end{proof}

\section*{Acknowledgements} 
Qi Zhou would like to thank David Damanik for posing the block version of
the problem considered here, and Jake Fillman for useful discussions and
for bringing the reference \cite{Levi} to our attention.
 Zhenfu Wang is supported by NSFC grant (124B2011) and Nankai Zhide Foundation. Bei Zhang is supported by Nankai Zhide Foundation. Qi Zhou is supported by NSFC grant (12531006, 12526201) and Nankai Zhide Foundation.

\end{document}